\documentclass[reqno]{elsarticle}

\usepackage{amsthm}

\usepackage{subcaption}
\usepackage{imakeidx} %Para indice alfanumerico
\usepackage{verbatim}
\usepackage{graphicx} %Para trabajar con imagenes
\usepackage{enumerate} %enumerar con numeros romanos en minúscula
\newtheorem{theorem}{Theorem}[section]
\newtheorem{proposition}[theorem]{Proposition}
\newtheorem{lemma}[theorem]{Lemma}

\newtheorem{corollary}[theorem]{Corollary} 
\theoremstyle{definition}

\newtheorem{definition}[theorem]{Definition}

\newtheorem{remarkx}[theorem]{Remark}
	{\popQED\\}{\endremarkx}

\def\bdem{\begin{proof}}
\def\edem{\end{proof}}
\def\bequ{\begin{equation}}
\def\eequ{\end{equation}}

\newcommand{\C}{\mathbb{C}}

\newcommand\I{\Omega}

\newcommand{\N}{\mathbb{N}}

\newcommand{\G}{\mathcal{G}}

\usepackage{todonotes}
\usepackage{verbatim}
\usepackage{graphicx} 
\usepackage{enumerate}
\usepackage{varwidth}
\usepackage{algpseudocode}
\usepackage{algorithm,algpseudocode}

\def\N{\mathbb{N}}

\def\C{\mathbb{C}}
\def\G{\mathcal{G}}
\def\AA{A}
\def\I{\mathbb{I}}
\def\diag{\text{diag}}
\def\Tr{\text{Tr}}
\usepackage{url}

\usepackage{makeidx}
\makeindex

\usepackage{amsmath,amssymb}

\begin{document}

\begin{frontmatter}

\title{Error analysis on the initial state reconstruction problem \tnoteref{t1}} 
\tnotetext[t1]{The research for this article has mainly been carried out while R. D\'\i az Mart\'\i n was  holding a postdoctoral fellowship from CONICET, and  I. Medri was holding postdoctoral fellowships from Vanderbilt University. }
\author[1]{Roc\'\i o D\'\i az Mart\'\i n \corref{cor1}}
\ead{rocio.diazmartin@unc.edu.ar}
\author[2]{Ivan Medri}
\ead{cocarojasjorgeluis@gmail.com}
\author[3]{Juliana Osorio %
}

\cortext[cor1]{Corresponding author}
\address[1]{IAM -- CONICET, Saavedra 15, CABA, Argentina and Universidad Nacional de C\'ordoba, Ciudad Universitaria, C\'ordoba, Argentina.}

\address[2]{Department of Mathematics, Vanderbilt University, Nashville, TN 37240-0001, USA. }

\address[3]{Dto. de Matem\'atica, FCEyN, Universidad de Buenos Aires, Ciudad Universitaria, CABA, Argentina}

%\date{}

\begin{keyword}
Dynamical Sampling\sep Control Theory\sep Kalman filter\sep Observability\sep Stability
%\MSC[2010]{93B07}
\end{keyword}

\begin{abstract} In this paper we propose a method to estimate the initial state of a linear dynamical system with noisy observation. The method allows the user to have estimations in real time, that is, to have a new estimation for each new observation. Moreover, at each step, the covariance matrix of the error is known and it is proved that the dynamic of the state estimator error is always Lyapunov stable. Also, %necessary and 
sufficient conditions are given to guarantee asymptotic stability for  the error dynamics of an LTI dynamical system, which is itself an LTV system.   
\end{abstract}

\end{frontmatter}

\section{Introduction}

The aim of this paper is to  recover the initial state of a given   dynamical system from noisy observations (and without noise in the dynamic) as well as to analyze the asymptotic stability of the method. 
%For discrete dynamical systems, if the dimension of the space of states is finite, one can find in the literature some algorithms to recover both the initial or the final state. 

%For discrete dynamical systems, if the dimension of the space of states is finite, there are some classic algorithms to recover the final state of a system and its corresponding asymptotic analysis in the case of noisy dynamics, but only recently the case of perfect dynamic has been addressed (see \cite{Zhang}).  Initial state recovery is often more unreliable in the case of noisy dynamics and the case of perfect dynamics can be now addressed following the ideas of the aforementioned article. 

%In particular, the communities of control and dynamical sampling theory have addressed the problem under certain conditions. 

Specifically, consider the discrete-time linear dynamical system 
\begin{equation}\label{dynSys}
    \begin{cases}
        x(k+1)=A_{k+1}x(n), & x(0)=x_0\\
        y(k) = H_k x(k) + v_k, & v_k \sim N(0, R_k) 
    \end{cases}
\end{equation}
where, for each $k\in\N\cup\{0\}$, $x(k)\in\C^d$ are the \textit{states} of the system, $y(k)\in \C^m$ are the \textit{observations}, and  $A_k$ and $H_k$ are  matrices in $\C^{d\times d}$ and $\C^{m\times d}$, respectively, called the \textit{dynamic operator} and the \textit{observation or sampling operator} at time $k$. It is assumed that $m<d$. If the dynamic and sampling operators are constant, we said that the system is \textit{linear time invariant} (LTI). The sequence $\{v_k\}$ is a random process with known distribution, typically assumed normal with mean zero and such that 
$\mathbb{E}(v_j v_k^*) = R_k\delta_{j,k}$,
where the matrices $R_k$ are symmetric positive definite.
In this paper, we deal with the \textit{observability} problem. That is, we will consider the dynamical system \eqref{dynSys} where the unknown initial state $x_0$ is to be recovered from the set $\{y(k)\}_k$.

Some methods for this problem, proposed in Control Theory, are known as \textit{Kalman filter methods} (see, e.g., \cite{Jaz,Simon}). They are applied for systems that include both,  
noise in the observations and in the dynamic.
Roughly speaking, they can be deduced by making a succession of estimates that come from solving a recursive weighted least square problem.
A priori, the development of Kalman filters does not focus on estimating the initial state by knowing  samples in later times, but is based on reconstructing the final state through the knowledge of previous observations. Both problems, the one of recovering the initial state and the one of recovering the final state, can be related by introducing fictitious variables that remain static over time, thus obtaining that the final state of the augmented system always contains the initial state (cf. \cite[Chapter 9]{Simon}). However, in presence of noise, this leads to the reconstruction problem of the initial state being less robust than that of the final state. A specific error control is then required. 
%Particularly, in the field of Control Theory the approaches that allow to find the final state of a dynamical system under noise in the dynamic and in the observations are usually called \textit{Kalman filter methods} \cite{Jaz}. 
 
 It is worth noting that in the 
 recent paper \cite{Zhang}, the classical stability analysis of Kalman filter methods  is reviewed to include dynamical systems where the dynamics are unperturbed as in \eqref{dynSys}. The study of the situation in absence of
 process noise in the state equation
 is useful for the cases where information about the models is more precise, and this is the framework we will consider in the present paper.

 For LTI dynamical systems, when the observation operator $H$ is presented as
\begin{equation*}\label{rdm-sist din}
        Hx=\{\langle x , g\rangle\}_{ g \in \G},
\end{equation*}
for some set of vectors $\G\subset \C^d$ (of cardinality less or equal than $d$), the observability problem is known as the \textit{dynamical sampling} problem.  The collection of the observations can be viewed as \textit{time-space samples} of the form
\begin{equation}\label{framee}
   \{y(k)\}_{0\leq k\leq L-1}=\{\left(\langle x(k) , g\rangle\right)_{ g \in \G,0\leq k\leq L-1}= \{\langle  A^kx_0 ,g\rangle\}_{ g \in \G,\, 0\leq k\leq L-1}.
\end{equation}
We said that we solved the dynamical sampling problem when \eqref{framee} becomes a \textit{frame} for $\C^d$ (for some final time $L$). 
The dynamical sampling problem has been widely studied (see, e.g., \cite{acmt, random-noise} and the references therein). For an explicit formulation of the connections between Dynamical Sampling and Control Theory we mention the recent article \cite{dmmm}. We remark that in the framework of dynamical sampling, non perturbation on the dynamic operator are considered.

The papers  
\cite{rdm-adk, random-noise} developed algorithms for solving the dynamical sampling problem where the space measurements are inexact. Furthermore, they treat the case where dynamic operator $A$ is unknown. The techniques behind these papers lie on least squares methods and  \textit{denoising} using  \textit{Cadzow} algorithms \cite{Cadzow}. However,  the authors make several assumptions on the dynamic operator. In particular, they analyze in great detail the case where it is a  circulant matrix. With this extra assumption Fourier analysis techniques are suitable. 
%Apart from that, we mention that inside the community of Dynamical Sampling  one can find, for example, the article \cite{random-noise}.

Our objectives are to propose a new way to estimate the initial state of a discrete-time dynamical system (without noise in the dynamics) by using \textit{a priori} information (noisy observations) %that could be adaptable to the case where the state space is infinite dimensional 
and to develop the stability analysis on the dynamic of the state estimator error.

We organize the paper as follows: In Section \ref{sec_prelim} we introduce some notation to be used for the rest of the paper. Then, we state what we are going to called a \textit{uniformly observable}  dynamical system by adapting the notion of \textit{uniformly completely observable} dynamical system given in  \cite{Jaz,Zhang}: Since our goal is to reconstruct the initial state (not the final state),
we translate the requirements on the so-called \textit{information matrix} (cf. \cite{Jaz}) to the \textit{observability matrix} (given in \eqref{obs_matrix}).
Apart from that, we include some preliminaries results from stability theory of dynamical systems. The main contributions of this paper start in Section \ref{sec_our_method} where we proposed our method to recover the initial state of a system. This approach is based on the Kalman methods used to recover final states. 
We particularly seek to provide a new scheme to those given from the side of dynamical sampling theory (see Remark \ref{remark1})  by introducing techniques from control theory. Finally, in Section \ref{asymptotics} we deal with the problem of stability of our method. For general linear time-variant dynamical systems we obtain Lyapunov stability for the dynamic of the state estimator  error  (Theorem \ref{teo_stability_a_secas}), whereas for LTI systems we proved asymptotic stability under an extra hypothesis on the eigenvalues of the dynamic matrix (Theorem \ref{asymp stability theorem}). Numerical examples are presented at the end of this section.

\section{Preliminaries}\label{sec_prelim}

\subsection{Notation and assumptions}\label{assumptions}
%We introduce some notation to be used for the rest of the paper.

For the dynamical system \eqref{dynSys}
the initial estate $x_0$ is the unknown but deterministic vector we desire to estimate and 
%We assume that the matrices $A(k)$ are non singular for all $k\geq 0$ and 
we use the following convention:
$$
A(k,k) = \mathbb{I}, ~~ A(k+1,k) = A_{k+1}
$$
where $\mathbb{I}$ denotes the identity matrix. For all $k\in \N\cup\{0\}$ the matrices $A_k$ are assumed to be invertible, so the dynamics is time-reversible, then if $j<k$
\begin{align*}
    A(k,j) =& \prod_{\ell = 0}^{k-j-1}A(k-\ell) = A_{k}A_{k-1}\cdots A_{j+2}A_{j+1} \quad \text{ and } \quad 
    A(j,k) = A^{-1}(k,j)
\end{align*}
The matrix $A(n,j)$ is called the \textit{state transition matrix} of the system and with this notation we can go from the state $x(j)$ at time step $j$ to the state $x(k)$ at time step $k$ via $x(k)=A(k,j)x(j)$. Then we have 
\begin{equation*}
\begin{cases}
    x(k) = A(k,0)x_0 \\
    y(k) = H_kA(k,0)x_0 + v_k
\end{cases}
\end{equation*}
Also, we denote
$$
\widetilde{H}_k = H_kA(k,0)
$$
%Note that we have used the index notation to define the time varying matrices $\widetilde{H}_k$. We will use the index for other than the transition matrices to ease the notation.
which is  the operator that observes the exact evolution of the initial state $x_0$ at time $k$.

For the case of a linear time invariant system (LTI) we have $A_{k} = A$ and $H_k = H$, therefore the transition matrices and  $\widetilde{H}_k$ take a simpler form:
$$
A(k,0) = A^k, \qquad \widetilde{H}_k = HA^k 
$$

Finally, regarding the noise $m\times m$ real covariances  $R_k$ we assume they are all positive definite matrices with a strictly positive lower bound which we will assume once and for all to be $\sigma^2\mathbb I$, i.e,
$$
R_k \geq \sigma^2 \mathbb I \qquad \forall k\in\N\cup\{0\}.
$$

\subsection{Observability}

One may say that \textit{to observe} the system \eqref{dynSys} means to recover the initial data $x_0$ by knowing only the output function $k \mapsto y(k)$. With more mathematical rigor, first notice that observations of the dynamical system \eqref{dynSys} are given by
\begin{equation}\label{obs2}
    \begin{bmatrix}
    y(0)\\
    y(1)\\
    \vdots\\
    y(L-1)
    \end{bmatrix}=\mathcal{A}_L \, x_0
    +\begin{bmatrix}
    v_0\\
    v_1\\
    \vdots\\
    v_{L-1}
    \end{bmatrix}
\end{equation}
where $\mathcal{A}_L$ is the $mL\times d$ block matrix
\begin{equation*}
    \mathcal{A}_L:=\begin{pmatrix}
    H_0\\
    H_1A(1,0)\\
   \vdots\\
    H_{L-1}A({L-1},0)
    \end{pmatrix}= \begin{pmatrix}
    \widetilde{H}_0\\
    \widetilde{H}_1\\
    \vdots\\
    \widetilde{H}_{L-1}
    \end{pmatrix}
\end{equation*}
Let $\mathcal{R}_L$ be the $mL\times mL$ diagonal block matrix
\begin{equation*}
    \mathcal{R}_L:=\begin{bmatrix}
    R_0& \\
    &R_1&\\
    & & \ddots\\
    & &  &R_{L-1}
    \end{bmatrix}
\end{equation*}
The observability problem consists of solving the a weighted least squares problem (with weight given by $\mathcal{R}_L^{-1}$) which reduces to see if $\mathcal{A}_L^*\mathcal{R}^{-1}_L\mathcal{A}_L$ is  positive defined,
which is equivalent to
\begin{equation}\label{obs_matrix}
    \mathcal{O}(L,0):=\sum_{j=0}^{L-1} \AA(j,0)^* H_j^* R_j^{-1}H_j\AA(j,0) > 0
\end{equation} 
The matrix $\mathcal{O}$ is called the \textit{observability matrix}.

\begin{definition}\label{obs def}
The dynamical system \eqref{dynSys} is said to be
%\begin{enumerate}[(i)]
%    \item \textit{Observable} if there exists $L\in\N$ such that
%\begin{equation*}
%    \mathcal{O}(L,0)=\sum_{j=0}^{L-1} \widetilde{H}_j^TR_j^{-1}\widetilde{H}_j >0.
%\end{equation*}
\textit{uniformly observable}  if there exists $L\in\N$ and $\rho>0$ such that
\begin{equation*}
    \mathcal{O}(k+L,k)=\sum_{j=k}^{k+L-1} A(j,k)^*H_j^*R_j^{-1}H_jA(j,k) \geq \rho\mathbb{I} \quad \text{for all } k\geq 0.
\end{equation*}
% OR EQUIVALENTLY 
%\begin{equation*}
%     \mathcal{O}(k,k-L)=\sum_{j=k-L}^{k-1} A(j,k-L)^TH_j^TR_j^{-1}H_jA(j,k-L) \geq \rho\mathbb{I} \quad \text{for all } k\geq L
% \end{equation*}
%     \item \textit{Super observable}  if there exist $L\in\N$ and $\rho>0$ such that
% \begin{equation*}
%     \sum_{j=k-L}^{k-1} \widetilde{H}_j^TR_j^{-1}\widetilde{H}_j \geq \rho\mathbb{I} \quad \text{for all } k\geq L
% \end{equation*}
%dadas L observaciones consecutivas, empezando desde cualquier tiempo, puedo recuperar el estado inical x_0 
%\end{enumerate}
\end{definition}

%While the concept of uniform complete observability itself is simple and straightforward to define, it is more often that not very difficult to show that a particular system is uniformly completely observable (UCO). This is a true limitation as the concept of uniform completely observability is instrumental in numerous studies. It permits to establish the convergence of the Kalman filter for LTV systems (see Besan¸con (2007) and references therein). It is also recurrent in adaptive control [...] página 2 del paper [Relaxed conditions for uniform complete observability and controllability of LTV systems with bounded realizations] 

For the case of an LTI dynamical system uniform observability is equivalent to \textit{observability}, that is, there exists $L>0$ such that the observability matrix is
\begin{equation*}
    \mathcal{O}(L,0)=\sum_{j=0}^{L-1} (A^{j})^*H^*R_j^{-1}HA^j .
\end{equation*}
is strictly positive definite.

\begin{remarkx}\label{remark1}
Without noise, solving the dynamical sampling problem is equivalent to have $\mathcal{A}_L$ full-rank. The \textit{frame operator}  $\mathcal{A}_L^*\mathcal{A}_L$ is the {\it observability matrix} $\mathcal{O}(L,0)$ associated to the LTI dynamical system.

We point out that in \cite{ random-noise}, assuming  that for some fixed number $L\geq 1$, the matrix $\mathcal{A}_N$ has full rank  for every $N\geq L$, and considering noisy measurements as in \eqref{obs2} with random noise 
having $\mathbb{E}(v_j)=0$ and the same covariance matrix $\sigma^2\mathbb{I}$ for all $j\in\N\cup\{0\}$ (for some $\sigma>0$),  the authors  propose the $N$-th estimation $\widetilde{x}_N$ of the initial state as
\begin{equation*}
    \widetilde{x}_N=\mathop{argmin}_{x\in\C^d} \sum_{j=0}^{N-1}\|HA^jx-y(j)\|%=\left(\mathcal{A}_N^*\mathcal{A}_N\right)^{-1}\mathcal{A}_N^* \, {\bf y}_N
\end{equation*} 
We will introduce a new method which allows an update with each new observation data.
\end{remarkx}

\subsection{Lyapunov Stability}

The algorithm proposed in this work to estimate the initial state of a given observable system will produce a new dynamical system for the estimation errors. Since our goal is to prove the stability for this time varying dynamical system, we take the opportunity to recall the different stages of stability and characterizations for asymptotic stability in the case of linear time varying systems. For a clarifying exposition on the subject see \cite{bof, haddad}.

%\textcolor{red}{Decidir si en los sistemas dinamicos las matrices del estado $k$ al $k+1$ son $A_k$ y $\Psi_k$ o $A_{k+1}$ y $\Psi_{k+1}$.}

\begin{definition}
Given the discrete dynamical system 
\begin{equation}\label{dyn_sist_aux}
x(k+1) = \Psi_{k+1}x(k), \qquad x(k_0) = x_0, \qquad k\geq k_0    
\end{equation}
with $\Psi_k$ a linear operator for all $k$, we say that the  zero solution (the only equilibrium point) of \eqref{dyn_sist_aux} is:
\begin{enumerate}[(i)]
    \item \textit{Lyapunov stable} if for every $\epsilon>0$ and $k_0\in \N$ there exists $\delta = \delta(\epsilon, k_0)>0$ such that $\|x_0\| < \delta$ implies that $\|x(k)\|<\epsilon$ for all $k\geq k_0$.
    \item \textit{Uniformly stable} if for every $\epsilon>0$ there exists $\delta = \delta(\epsilon)>0$, independent on $k_0$ such that $\|x_0\| < \delta$ implies that $\|x(k)\|<\epsilon$ for all $k\geq k_0$.
    \item \textit{Asymptotically stable} if it is Lyapunov stable and for every $k_0\in \N$ there exists $\delta = \delta(k_0)>0$ such that $\|x_0\|<\delta$ implies $\lim_{k\to\infty} x(k) = 0$.
    \item \textit{Uniformly asymptotically stable} if it is uniformly stable and there exists $\delta>0$, independent of $k_0$, such that $\lim_{k\to \infty}\|x(k)\| =0$ uniformly in $k_0$ for all $\|x_0\|<\delta$.
    \item {\textit{Globally uniformly asymptotically stable} if for every  $x_0$,  $\lim_{k\to\infty} \|x(k)\| = 0$.}
    %\textcolor{blue}{Globally uniformly asymptotically stable: if it is uniformly stable, $\delta(\epsilon)$ can be chosen to satisfy $lim_{\epsilon\to\infty}\delta(\epsilon)= \infty$, and, for each pair of positive numbers $\eta$ and $c$, there is $T = T(\eta, c) > 0$ such that $\|x(k)\| < \eta$ for all $k >k_0  + T(\eta, c)$ and $\|x(k_0)\| < c$.}
\end{enumerate}
\end{definition}

\begin{proposition}(cf. \cite{bof})\label{el-lemma!!} Let $\Psi(k,k_0)$ denote the state transition matrix for \eqref{dyn_sist_aux} from time $k_0$ to time $k$, then the following are equivalent:
\begin{enumerate}
    \item The equilibrium point of system \eqref{dyn_sist_aux} is uniformly asymptotically stable
    \item   $\lim_{k\to\infty}\|\Psi(k,k_0)\|= 0$ uniformly in $k_0$.
    \item $\|\Phi(k,k_0)\| \leq \alpha e^{-\lambda(k-k_0)} \ \ \forall k\geq k_0$ for some positive constants $\alpha, \lambda$.
    \item {The equilibrium point of the system \eqref{dyn_sist_aux} is globally uniformly asymptotically stable.}
\end{enumerate}
This in turn says that in the case of linear systems asymptotic stability and \textit{exponential stability} are equivalent.

\end{proposition}

\section{Kalman-based method for initial state recovering}\label{sec_our_method}

In this section we develop a recovery algorithm for a general linear dynamical system in the presence of noisy observations with known statistics as in \eqref{dynSys}.
% \begin{equation}\label{dynSys2}
%     \begin{cases}
%         x(n+1)=A_{n+1}x(n), \qquad x(0)=x_0\\
%         y(n) = H_nx(n) + v_n \qquad v_n \sim N(0, R_n) \ \ \forall n\in\N 
%     \end{cases}
% \end{equation}

The idea of this method to estimate the initial state $x_0$ is the following: We start from an initial guess denoted by $\widehat{x}_0$ and for each time step we improve the estimation by correcting the previous one with the so called \textit{innovation term}, in this way for the next estimation we take into account the last available observation of the system. More explicitly, the estimator is of the form
\begin{align}\label{initial_state_estimator}
    \widehat{x}_k =& \widehat{x}_{k-1} + K_k\left(y(k-1) - \widetilde{H}_{k-1}\widehat{x}_{k-1}\right) \qquad k =1,2,\dots\\
    \widehat{x}_0 =& \text{initial guess} \notag
\end{align}
where $\widehat{x}_k$ denotes the $k$-th estimation of the initial state $x_0$. 
The matrices $K_k$ are to be defined according to some optimality criteria. 

Note that the $k$-th estimation of the initial state ($\widehat{x}_k$)  takes into account all the first $k$ noisy observations $\{y(n)\}_{n=0}^{k-1}$.

Set $e_k=\widehat{x}_k-x_0$ the error given by the $k$-th estimation of the initial state $x_0$. Its covariance matrix is given by
\begin{equation}\label{covariance}
    P_k:=\text{Cov}(e_k,e_k)=\mathbb{E}\left[(\widehat{x}_k-x_0)(\widehat{x}_k-x_0)^*\right] - \mathbb{E}\left[(\widehat{x}_k-x_0)\right]\mathbb{E}\left[(\widehat{x}_k-x_0)\right]^*
\end{equation}
For each step $k$, the matrix $K_k$ is chosen such that it minimizes the following energy functional 
\begin{equation*}
J(\widehat{x}_k)=\frac{1}{2}\mathrm{Tr}(P_k).    
\end{equation*}
The trace of the covariance is called the variance of the error estimation. In this sense we are defining our estimator as having minimum variance. The next theorem gives a prescription of the algorithm in terms of the initial guess and its initial covariance error.

\begin{theorem}\label{algorithm}
Given the dynamical system \eqref{dynSys}, the solution of the minimization problem given above is
\begin{equation}\label{Kalman_init_state}
    K_k=P_{k-1}\widetilde{H}_{k-1}^*\left( \widetilde{H}_{k-1} P_{k-1} \widetilde{H}_{k-1}^* + R_{k-1}\right)^{-1}.
\end{equation}
and the actualization of the covariance error matrix is given by
\begin{equation}\label{P_k}
    P_k=P_{k-1} - K_k\widetilde{H}_{k-1}P_{k-1}
\end{equation}
with $P_0$ the initial guess for the covariance error.
\end{theorem}

\begin{proof}
From equations \eqref{dynSys} and  \eqref{initial_state_estimator}  we have a recursive formula for the error
\begin{align}\label{error}
    e_k &= \widehat{x}_k - x_0= (\I-K_k\widetilde{H}_{k-1})\widehat{x}_{k-1} + K_k{y}(k-1) - x_0\notag\\
    &= (\I-K_k\widetilde{H}_{k-1})\widehat{x}_{k-1} + K_k(\widetilde{H}_{k-1} x_0 + {v}_{k-1}) - x_0\notag\\
    &= (\I-K_k\widetilde{H}_{k-1})e_{k-1} + K_k{v}_{k-1} 
\end{align}
Therefore the covariance matrix of the error at the $k$-th estmation is
    \begin{align}\label{P_aux1}
    P_k=& \mathbb{E}\left[((\I-K_k\widetilde{H}_{k-1})e_{k-1} + K_k{v}_{k-1})((\I-K_k\widetilde{H}_{k-1})e_{k-1} + K_k{v}_{k-1})^*\right]- \notag\\
    & \mathbb{E}\left[(\I-K_k\widetilde{H}_{k-1})e_{k-1} + K_k{v}_{k-1}\right]\mathbb{E}\left[(\I-K_k\widetilde{H}_{k-1})e_{k-1} + K_k{v}_{k-1}\right]^* \notag\\
    =& (\I-K_k\widetilde{H}_{k-1})\mathbb{E}(e_{k-1}e_{k-1}^*)(\I-K_k\widetilde{H}_{k-1})^* + K_k\mathbb{E}(v_{k-1}v_{k-1}^*)K_k^*+\notag\\
    &(\I-K_k\widetilde{H}_{k-1})\mathbb{E}(e_{k-1}{v}_{k-1}^*)K_k^* + K_k\mathbb{E}({v}_{k-1}e_{k-1}^*)(\I-K_k\widetilde{H}_{k-1})-\notag\\
    &
     (\I-K_k\widetilde{H}_{k-1})\mathbb{E}(e_{k-1})\mathbb{E}(e_{k-1}^*)(\I-K_k\widetilde{H}_{k-1})^* \notag\\
     =& (\I-K_k\widetilde{H}_{k-1})P_{k-1}(\I-K_k\widetilde{H}_{k-1})^*+K_k{R}_{k-1}K_k^* 
    \end{align}
where we have used that $\mathbb{E}({v}_{k-1}e_{k-1}^*)=0$ since  $v_{k-1}$ and $e_{k-1}$ rise from different times, in fact $v_{k-1}$ is the noise of the $(k-1)$-th observation $y(k-1)$ and $e_{k-1}$ is the error of the $(k-1)$-th estimation of the initial state which takes into account the first $(k-1)$ observations starting from $0$, that is $\{y(0),y(1),\dots, y(k-2)\}$.

Then, the linear functional that we have to minimize  is given by
\begin{align*}
    J(K_k)&=\frac{1}{2}\mathrm{Tr}(P_k)\\
    &=\frac{1}{2}\mathrm{Tr}\left((\I-K_k\widetilde{H}_{k-1})P_{k-1}(\I-K_k\widetilde{H}_{k-1})^*+K_k{R}_{k-1}K_k^*\right)\\
    &=\frac{1}{2}\mathrm{Tr}\left(P_{k-1}\right)-\mathrm{Tr}\left(K_k\widetilde{H}_{k-1}P_{k-1}\right) +\\
    &\, \, \, \, \, \, \, \,  \frac{1}{2}\mathrm{Tr}\left(K_k(\widetilde{H}_{k-1}P_{k-1}\widetilde{H}_{k-1}^* + R_{k-1} )K_k^*\right)
\end{align*}
Imposing that 
\begin{equation*}
    \frac{\partial J}{\partial K_k} = - \left(\widetilde{H}_{k-1}P_{k-1}\right)^* + K_k\left(\widetilde{H}_{k-1}P_{k-1}\widetilde{H}_{k-1}^* + R_{k-1}\right)=0
\end{equation*}
we obtain $K_k$ as in \eqref{Kalman_init_state}.

Finally, replacing \eqref{Kalman_init_state} in \eqref{P_aux1} we have 
\begin{align}\label{obs_s}
    P_k %&= (\I-K_k\widetilde{H}_{k-1})P_{k-1}(\I-K_k\widetilde{H}_{k-1})^*+K_k{R}_{k-1}K_k^* \notag \\
    &= P_{k-1} + K_k\left(\widetilde{H}_{k-1}P_{k-1}\widetilde{H}_{k-1}^* + R_{k-1}\right)K_k^* -K_k\widetilde{H}_{k-1}P_{k-1}-P_{k-1}\widetilde{H}_{k-1}^*K_k^* \notag \\
    &= P_{k-1} +P_{k-1}\widetilde{H}_{k-1}^*K_k^* -K_k\widetilde{H}_{k-1}P_{k-1}-P_{k-1}\widetilde{H}_{k-1}^*K_k^*\notag \\
    &%=P_{k-1} - K_k\widetilde{H}_{k-1}P_{k-1} 
    = (\mathbb{I} - K_k\widetilde{H}_{k-1})P_{k-1}
\end{align}
\end{proof}

This theorem gives raise to the algorithm described as follows.%in Algorithm \ref{alg-Kalman}. 

\begin{algorithm}[H] 
\caption{Kalman-based algorithm to recover $x_0$}
\label{alg-Kalman}
\begin{algorithmic}
\Require{$A\in \mathbb R^{d\times d}$, $H\in\mathbb{R}^{m\times d}$, $R_k$  Dynamic and observation operators, and the covariances for the noise. \\   $\quad \quad \, \, \,  T_f$ final time} 
\Ensure{$\widehat{x}_0$ the initial state estimation after $T_f$ steps. }

\Statex

{\textbf{Initialization:}}
$\widehat{x}_0$, $P_0$ initial guesses for the state and covariance error.\\

\Statex 
$\widetilde{H}_0 = H$
\Statex
    \For{$k = 1$ to $T_f$}                    
        \State {$K_k$ $\gets$ {$P_{k-1}\widetilde{H}_{k-1}^*(\widetilde{H}_{k-1}P_{k-1}\widetilde{H}_{k-1}^* + R_{k-1})^{-1}$}}
        \State {$\widehat{x}_k$ $\gets$ {$\widehat{x}_{k-1} + K_k(y_{k-1} - \widetilde{H}_{k-1}\widehat{x}_{k-1})$}}
        \State {$P_k \gets P_{k-1} - K_k\widetilde{H}_{k-1}P_{k-1}$}
        \State {$\widetilde{H}_k \gets \widetilde{H}_{k-1}A$} 
    \EndFor
    \State \Return {$\widehat{x}_{T_f}$}
\end{algorithmic}
\end{algorithm}

\section{On the stability of the state estimation error dynamics}\label{asymptotics}

As can be deduced from \eqref{error} and  \eqref{obs_s}, the dynamics for $\mathbb{E}(e_k)$ and for the matrices $P_k$ are the same. Indeed,
\begin{equation*}\label{expected value}
    \mathbb{E}(e_k) =\Psi_k\mathbb{E}(e_{k-1}) + K_k\mathbb{E}(v_{k-1})= \Psi_k\mathbb{E}(e_{k-1}) \qquad \text{ and } \qquad P_k=\Psi_k P_{k-1}
\end{equation*}
where the transitions are given by time varying matrix 
\begin{equation}\label{Psi}
 \Psi_k = \mathbb{I} - K_k\widetilde{H}_{k-1}   
\end{equation}
for $K_k$ are computed according to  \eqref{Kalman_init_state} in Theorem \ref{algorithm}.
In this section we look for the asymptotic stability of a linear time-varying system of the form
\begin{equation}\label{eq_estabilidad_1}
    z(k)=\Psi_kz(k-1), \qquad z(0) =z_0.
\end{equation}
with $\Psi_k$ as in \eqref{Psi}, which corresponds to the dynamics of the expectation of the state dynamic error
and of the covariance matrices of the proposed estimation method for the initial state of a system as in \eqref{dynSys}.
% where $\Psi_k = \mathbb{I} - K_{k}\widetilde{H}_{k-1}$ and matrices $K_k$ are prescribed in Theorem \ref{algorithm}.
% From \eqref{error} the dynamic for the error of the estimator is given by
% \begin{equation}\label{error_dyn}
%     e_k =\Psi_k e_{k-1} + K_k v_{k-1}, 
% \end{equation}

%If we want to analyse its stability we only need to see the homogeneous part. Explicitly we are going to explore the stability of the equation \eqref{eq_estabilidad_1}. 
% What is more, if we look at the expectation in \eqref{error_dyn} we have
% \begin{equation}\label{expected value}
%     \mathbb{E}(e_k) =\Psi_k\mathbb{E}(e_{k-1}) + K_k\mathbb{E}(v_{k-1})= \Psi_k\mathbb{E}(e_{k-1})
% \end{equation}
% that is,  \eqref{eq_estabilidad_1} corresponds to the dynamic of the expectation of the state dynamic error.
%\textcolor{blue}{(*)}

The next proposition summarizes the properties of the transition matrices.
%\textcolor{blue}{quizás este párrafo debería ir en (*)}

\begin{proposition}\label{Psi props}
The matrices $\Psi_k$ satisfy the following:
\begin{itemize}
    \item[i)] $\Psi_k = P_kP_{k-1}^{-1}$ for all $k\geq 1$.
    \item[ii)] $\Psi(k,j) = P_kP_j^{-1}$ for all $k\geq j$. Therefore $\Psi(j,k) = P_jP_k^{-1}$
\end{itemize}
\end{proposition}
\begin{proof}
i) follows by \eqref{obs_s}. To prove $ii)$ we simply write the transition matrix and use that $\Psi(i,i-1) = \Psi_i$:
\begin{align*}
    \Psi(k,j) =& \Psi(k,k-1)\Psi(k-1,k-2)\Psi(k-2,k-3)~\cdots~ \Psi(j+2,j+1)\Psi(j+1,j) \\
    = & P_kP_{k-1}^{-1}P_{k-1}P_{k-2}^{-1}~\cdots~ P_{j+2}P_{j+1}^{-1}P_{j+1}P_j^{-1} = P_kP_j^{-1}.
\end{align*}
Finally, by item $i)$ $\Psi_k$ is invertible as it is the product of two non singular matrices, therefore $\Psi_k^{-1} = P_{k-1}P_k^{-1}$. Then,
$$
\Psi(j,k)=\Psi(k,j)^{-1} = P_jP_k^{-1}
$$

\end{proof}

\begin{proposition}\label{Pk_bound}
If the covariance matrix $P_0$ of the initial guess is strictly positive, 
then $P_k$ is upper bounded for all $k\geq 0$. Moreover, 
\begin{equation}\label{morover}
  P_0\geq P_1\geq P_2\geq \dots\geq P_k\geq P_{k+1}\geq \dots
\end{equation}
and their inverses satisfy the relation
\begin{equation}\label{eq-P_k}
  P_{k}^{-1}=P_0^{-1}+\mathcal{O}(k,0)  
\end{equation}
\end{proposition}

\begin{proof}
Replacing \eqref{Kalman_init_state} in \eqref{P_k} we have
\begin{align}\label{P_aux2}
    P_k%&=P_{k-1} -K_{k}\widetilde{H}_{k-1}P_{k-1} \notag\\
       &=P_{k-1} -P_{k-1}\widetilde{H}_{k-1}^*\left(\widetilde{H}_{k-1}P_{k-1}\widetilde{H}_{k-1}^* +R_{k-1}
       \right)^{-1}\widetilde{H}_{k-1}P_{k-1}.
\end{align}
Take the matrix inverses of both sides of \eqref{P_aux2} while applying the matrix equality
\begin{equation*}
    (G+V^*WV)^{-1}= G^{-1} - G^{-1}V^*\left(W^{-1} +VG^{-1}V^* \right)^{-1}VG^{-1}
\end{equation*}  with
\begin{equation*}
    \begin{cases}
    G=P_{k-1}\\
    V=\widetilde{H}_{k-1}P_{k-1}\\
    W=-\left(\widetilde{H}_{k-1}P_{k-1}\widetilde{H}_{k-1}^* +R_{k-1}
       \right)^{-1}
    \end{cases}
\end{equation*}
we obtain a recursion for $P_{k}^{-1}$
\begin{equation*}
    P_{k}^{-1}=P_{k-1}^{-1}+\widetilde{H}_{k-1}^*R^{-1}_{k-1}\widetilde{H}_{k-1}\geq P_{k-1}^{-1} 
\end{equation*}
From that recursion we obtain  \eqref{morover}.
    Moreover, it can be seen that
    \begin{align*}
        &P_0^{-1}>0  \qquad \text{ (by hypothesis)}  \\
             &P_{1}^{-1}= P_0^{-1} + H^*R_0^{-1}H%\\
             %&P_{k}^{-1}= P_0^{-1} + \sum_{j=0}^{k-1}\widetilde{H}_j^*R_j^{-1}\widetilde{H}_j 
    \end{align*}
    and recursively we obtain
    $$P_{k}^{-1}= P_0^{-1} + \sum_{j=0}^{k-1}\widetilde{H}_j^*R_j^{-1}\widetilde{H}_j,$$
    that is, 
    \begin{equation*}%\label{eq-P_k}
        P_{k}^{-1}=P_0^{-1}+\mathcal{O}(k,0)\geq P_0^{-1}  \qquad \text{ for all } k\geq 0.
    \end{equation*}
Finally, since $P_0$ is symmetric positive definite, we have that $P_0 \leq \rho_0 \mathbb{I}$ with $\rho_0 = \|P_0\|$, therefore 
$$
P_k^{-1} \geq \frac{1}{\rho_0}\mathbb{I} \qquad \text{ and } \qquad P_k \leq {\rho_0}\mathbb{I}.
$$
\end{proof}

\subsection{Lyapunov stability}

Lyapunov's direct method \cite{haddad} is one of the most popular ways to show all kind of stages of stability for a given dynamical system. In our case we were able to prove \textit{plain stability} for the estimation error dynamics considering as a candidate for Lyapunov function $V(k,z(k)) = z(k)^*P_k^{-1}z(k)$ with $P_k$ the error covariance matrices described in the previous proposition. 

\begin{theorem}\label{teo_stability_a_secas}
Let $\Psi_k$ as in \eqref{Psi}. If the initial covariance matrix $P_0$  is strictly positive then the zero solution of the system 
\begin{equation*}
\begin{cases}
z(k) = \Psi_kz(k-1)\\
z(0) = z_0        
\end{cases}
\end{equation*}
is stable.
\end{theorem}
\begin{proof}
To prove the stability it is enough to give a function $V(k,z(k))$ satisfying the following (see \cite[Theorem 13.11]{haddad}): 
$$\begin{cases}
    V(k, 0) = 0  \\
    V(k,z(k)) \geq \alpha(\|z(k)\|) \quad    \\
    \Delta V(k,z) := V(k,z(k)) - V(k-1, z(k-1)) \leq 0 
\end{cases}$$
\text{for some strictly increasing function } $\alpha:[0,\infty)\to [0,\infty)$.

\noindent Consider
\begin{equation}\label{lyapunov_func}
    V(k,z(k))=z(k)^*P_{k}^{-1}z(k)
\end{equation}
as our Lyapunov function. By Proposition \ref{Pk_bound} we have
\begin{equation*}
    V(k,z(k))\geq \frac{1}{\rho_0}\|z(k)\|^2 \qquad \forall k\geq0.
\end{equation*}
Now 
\begin{align*}
    V(k,z(k)) &= z(k)^*P_{k}^{-1}z(k)\\
    &=z(k-1)^*\left(\mathbb{I}-K_k\widetilde{H}_{k-1}\right)^* P_k^{-1} \left(\mathbb{I}-K_k\widetilde{H}_{k-1}\right)z(k-1) \qquad \text{ (by \eqref{eq_estabilidad_1})} \\
    &=z(k-1)^*\left(\mathbb{I}-K_k\widetilde{H}_{k-1}\right)^* P_k^{-1}P_kP_{k-1}^{-1}z(k-1) \qquad \text{ (by \eqref{obs_s})}\\
    &=V(k-1,z(k-1))-z(k-1)^*\widetilde{H}_{k-1}^*K_k^* P_{k-1}^{-1}z(k-1)\\
    &=V(k-1, z(k-1))-\\
    & \quad z(k-1)^*\widetilde{H}_{k-1}^*\left(\widetilde{H}_{k-1}P_{k-1}\widetilde{H}_{k-1}^*+R_{k-1}\right)^{-1}\widetilde{H}_{k-1}z(k-1) \, \, \,  \, \, \,  \text{ (by \eqref{Kalman_init_state})}.
\end{align*}
Hence
\begin{equation*}
    \Delta V(k,z) =  -z(k-1)^*\widetilde{H}_{k-1}^*\Sigma_{k-1}^{-1}\widetilde{H}_{k-1}z(k-1) \leq 0
\end{equation*}
where
\begin{equation*}\label{Sigma}
    \Sigma_{k-1}=\widetilde{H}_{k-1}P_{k-1}\widetilde{H}_{k-1}^*+R_{k-1}.
\end{equation*}
\end{proof}

\begin{remarkx}
Note that with the Lyapunov function given in \eqref{lyapunov_func} one cannot prove the asymptotic stability for the zero solution, one of the reasons being that the matrices $\widetilde{H}_{k}^*\Sigma_{k}^{-1}\widetilde{H}_k$ are only positive semi-definite.
\end{remarkx}

\subsection{Asymptotic stability for the error dynamics of an LTI dynamical systems}

From now on, we will consider the system \eqref{dynSys} to be LTI.
Before we state our main result about asymptotic stability, we introduce some notation and enunciate two important results concerning eigenvalues and singular values of a given matrix.
Given any $B\in \C^{d\times d}$, let 
\begin{equation*}
    \lambda_{\max}(B)= \max\{|\lambda|: \lambda \text{ eigenvalue of } B\}, \lambda_{\min}(B)= \min\{|\lambda|: \lambda \text{ eigenvalue of } B\}
\end{equation*}
Also, $s_i(B)$ denotes the $i$-th singular value and $s_{\min}(B), s_{\max}(B)$ denote the smallest and largest singular values respectively.
For $B\in \C^{d\times d}$ hermitian  denote  $\{\lambda_i(B)\}$ the set of eigenvalues of $B$ ordered by 
$$
\lambda_1(B) \geq \lambda_2(B) \geq\cdots \geq \lambda_d(B).
$$
If $C\in\C^{d\times d}$ is also hermitian then by 
Weyl's Theorem \cite{Horn} we have
\begin{equation}\label{weyl}
    \lambda_i(B)+\lambda_{d}(C)\leq\lambda_i(B+C)\leq\lambda_i(B)+\lambda_{1}(C).
\end{equation}
Finally, a Gelfand type result for the asymptotic behaviour of the singular values of the powers of a matrix. 
\begin{theorem} (cf. \cite{yamamoto})\label{gelfand singular values}
For any given matrix $A\in \mathbb{C}^{d\times d}$ with eigenvalues $\lambda_{k}(A)$, $k=1,\dots, d$  ordered by their absolute values in a non increasing way $|\lambda_1|\geq |\lambda_2|\geq\cdots\geq|\lambda_d|$. For any $n\in\mathbb{N}$ denote $s_i(A^n)$ the singular values of $A^n$  also ordered non increasingly $s_1(A^n)\geq s_2(A^n)\geq\cdots\geq s_d(A^n)$ then
\begin{equation*}
    \lim_{n\to\infty}(s_i(A^n))^{1/n} = |\lambda_i|.
\end{equation*}

\end{theorem}

The following lemma describes the asymptotic behaviour of the smallest eigenvalue of the observability matrix. More explicitly, we show that the non-decreasing sequence $\{\lambda_{\min}\mathcal{O}(k,0)\}_{k\in\N}$ either goes to infinity or it is bounded, depending on the locus of the eigenvalues of the dynamic operator $A$.

\begin{lemma}\label{lambda_min O}
Let the system \eqref{dynSys} be observable and with dynamics given by $A$. Then the following holds:  
\begin{enumerate}
    \item  If  $\lambda_{\min}(A)> 1$, then
\begin{equation*}\label{O-no-acot}
    \lim_{k\to \infty}\lambda_{\min}\left(\mathcal{O}(k,0)\right)=\infty.
\end{equation*}
    \item If $\lambda_{\min}(A)<1$, then the sequence $\{\lambda_{\min}(\mathcal{O}(k,0)\}_{k\in\N}$ converges to a positive constant.
\end{enumerate}
\end{lemma}
\begin{proof}
(1) Let $L\in \N$ and $\rho>0$ given by the observability condition as in Definition \ref{obs def}. We will show that the subsequence $\{\lambda_{\min}(\mathcal{O}(nL,0))\}_{n\in\N}$ goes to infinity. 
We have that
\begin{equation*}
\mathcal{O}(nL,0)=\sum_{j=0}^{n-1}(A^*)^{jL} \, \mathcal{O}((j+1)L,jL) \, A^{jL}\geq \rho\sum_{j=0}^{n-1} (A^*)^{jL}A^{jL} \qquad \text{ for all } n\in\N.    
\end{equation*}
%In particular we have $\lambda_{\min}(\mathcal{O}(nL,0))\geq n\rho$. 
Let us denote
\begin{equation*}
   {\mathcal{O}}_{A,L}(n,0):= \sum_{j=0}^{n-1} (A^*)^{jL}A^{jL}.
\end{equation*}
Therefore,
\begin{equation*}
    \lambda_{\min}\left(\mathcal{O}(nL,0)\right)\geq\rho \lambda_{\min}\left(\mathcal{O}_{A,L}(n,0)\right).
\end{equation*}
%Notice that, for each $n\in\N$,is the observability matrix of an LTI system like \eqref{dynSys2} with dynamics $A^L$ and $H=\mathbb{I}$. 
Then,
\begin{align*}
   \lambda_{\min}\mathcal{O}(nL,0) &\geq \rho\sum_{j=0}^{n-1}\lambda_{\min}((A^*)^{jL}A^{jL})\\
   &\geq \rho \lambda_{\min}((A^*)^{(n-1)L}A^{(n-1)L}) = \rho ~s_{\min}^2(A^{(n-1)L})
\end{align*}
where in the first inequality we have used Weyl's inequality \eqref{weyl}. 
By Theorem \ref{gelfand singular values} we can take $\epsilon$ sufficiently small such that
\begin{equation*}
    s_{\min}^2(A^{(n-1)L}) >  (\lambda_{\min}^{L}(A) - \epsilon)^{2(n-1)}
\end{equation*}
with $\lambda_{\min}^{L}(A) -\epsilon >1$ and we get the result.

Furthermore, for $k$ sufficiently large we can write $k = nL + m$ for $m<L$, then 
\begin{equation}\label{exponential bound for O}
    \lambda_{\min}(\mathcal{O}(k,0))\geq \lambda_{\min}(\mathcal{O}(nL,0)) \geq \rho e^{\beta k} 
\end{equation}
with $\beta$ some positive constant depending on $L,\lambda_{\min}(A)$. This exponential order will be useful for our main result.
\medskip

For the case of a normal operator $A$ the result is rather straightforward and it also includes the case $\lambda_{\min}(A)=1$. In fact, given the eigen-decomposition of $A=U\diag(\lambda_1,\lambda_2,\dots,\lambda_d)U^*$ we have
\begin{equation*}
     \mathcal{O}_{A,L}(n,0)=U\begin{pmatrix}
     \sum_{j=0}^{n-1}|\lambda_1|^{2jL} &0 &\dots &0\\
     0&\sum_{j=0}^{n-1}|\lambda_2|^{2jL} & \dots &0\\
     \vdots&\vdots &\ddots&\vdots\\
     0&0&\dots &\sum_{j=0}^{n-1}|\lambda_d|^{2jL}
     \end{pmatrix}U^{*}
 \end{equation*}

\bigskip

(2) Since $\mathcal{O}(k,0)$ is symmetric and positive definite for $k\geq L$ we now that
    $$
        0<\lambda_{\min}(\mathcal{O}(k,0)) = \min_{\|x\|=1}\langle\mathcal{O}(k,0)x,x\rangle
    $$
Let ${\bf v}$ a (normalized) eigenvector of $A$ corresponding to the eigenvalue of minimum absolute value. Then
\begin{equation*}
    \lambda_{\min}(\mathcal{O}(k,0)) \leq \sum_{j=0}^{k-1}\langle (A^*)^{j}H^* R_j^{-1}HA^j{\bf v}, {\bf v}\rangle \leq~\frac{\|H\|^2}{\sigma^2}~\sum_{j=0}^{k-1}\lambda_{\min}(A)^{2j}<\infty
\end{equation*}
with $\sigma^2$ a positive lower bound for the noise covariances given in Section \ref{assumptions}.
%By the hipothesis of observability we know that $H{\bf v}\neq 0$. Since $\lambda_{min}(A)<1$ we get the result.
\end{proof}

\begin{theorem}\label{asymp stability theorem}
Consider the observable dynamical system \eqref{dynSys}. Let $\Psi_k$  as in \eqref{Psi}. If $\lambda_{\min}(A)>1$, then the equilibrium point of
\begin{equation*}
\begin{cases}
z(k) = \Psi_kz(k-1)\\
z(0) = z_0        
\end{cases}
\end{equation*}
is uniformly asymptotically stable. 
On the other hand, if $\lambda_{\min}(A)<1$ then the equilibrium point is not asymptotically stable.
\end{theorem}
\begin{proof}
In order to prove the result let us observe the following:

By Proposition \ref{Psi props} and equation \eqref{eq-P_k} we have
$$\Psi(k,k_0)=P_kP_{k_0}^{-1} \qquad \forall k\geq k_0
$$
with 
$$
P_n^{-1}=P_0^{-1}+\mathcal{O}(n,0) \quad\text{for every } n\in \mathbb{N}.
$$
Thus, by Proposition \ref{el-lemma!!} it is enough to show that there exist $\alpha, \beta>0$ such that
$$
\|\Psi(k, k_0)\|\leq\|P_k\| \|P_{k_0}^{-1}\|\leq \alpha e^{-\beta(k-k_0)}\quad \forall k\geq k_0\geq 0.
$$
Since the matrices $P_k$ are symmetric positive definite, by the Weyl's Theorem mentioned before we have the handy inequality for $\|P_k\|=\lambda_{\max}(P_k)$:
\begin{equation}\label{inequality norm P_k}
 \frac{1}{\lambda_{\max}\left(P_0^{-1}\right)+\lambda_{\min}\left(\mathcal{O}(k,0)\right)}\leq\lambda_{\max}(P_k)\leq\frac{1}{\lambda_{\min}\left(P_0^{-1}\right)+\lambda_{\min}\left(\mathcal{O}(k,0)\right)}   
\end{equation}
%\begin{equation*}
%    \lambda_{\max}(P_k) \geq \frac{1}{\lambda_{\min}\left(P_0^{-1}\right)+\lambda_{\max}\left(\mathcal{O}(k,0)\right)}
%\end{equation*}

First, for a fixed $k_0\in \mathbb{N}$, 
$$
\|P_{k_0}^{-1}\|\leq \|P_0^{-1}\| + \|\mathcal{O}(k_0,0)\| \leq \|P_0^{-1}\| +\frac{\|H^*H\|}{\sigma^2}\sum_{j=0}^{k_0-1}s_{\max}(A)^{2j} \leq Ce^{\beta' k_0}
$$
with $C,\beta'$ positive constants depending on $\|H^*H\|, \|P_0^{-1}\|,\sigma$ and $s_{\max}(A)$. 

Now, by the right hand inequality in \eqref{inequality norm P_k} and the exponential bound \eqref{exponential bound for O}  obtained in Lemma \ref{lambda_min O} we get
$$
\|P_k\| \leq \frac{1}{\rho}e^{-\beta k}.
$$
Therefore 
\begin{equation*}
    \|\Psi(k, k_0)\|\leq\alpha e^{-\beta(k-k_0)}\quad \forall k\geq k_0\geq 0,
\end{equation*}
with $\alpha, \beta$ positive constants depending on $L, \rho, \sigma^2, \lambda_{\min}(A),s_{\max}(A), \|H^*H\|$ and $\|P_0^{-1}\|$. 

\medskip

For the second assertion of the theorem observe that, due to the left hand inequality \eqref{inequality norm P_k} and Lemma \ref{lambda_min O}
$$
\lim_{k\to\infty} \|P_k\| \neq 0.
$$
But, as it was stated before in \eqref{P_k}, the matrices $P_k$ follow the dynamics $P_k = \Psi_kP_{k-1}$ with initial condition at time $k=0$ given by $P_0$. This in turn implies, by Proposition \ref{el-lemma!!}, that $\lim_{k\to\infty}\|\Psi(k,0)\|\neq 0$ which is equivalent to say that the equilibrium point is not asymptotically stable.
\end{proof}

\begin{remarkx}\
\begin{enumerate}
    \item If  $\lambda_{\max}(A)<1$, it can be easily seen that the equilibrium point of the system \eqref{eq_estabilidad_1}
is uniformly stable. 
    \item If the dynamics operator $A$ is normal, then $\lambda_{\min}(A)\geq 1$ is a necessary and sufficient condition for the uniform asymptotic stability of the estimation error dynamics.
\end{enumerate}
\end{remarkx}

The next result enlighten us on the importance of the previous theorem. First, we mention a very agreeable property of a given estimator. 

An estimator is said to be \textit{consistent} if, roughly speaking, has the property that as the number of observations (data) used increases indefinitely, the resulting sequence of estimates \textit{converges in probability} to the quantity to be estimated. A sufficient criterion for an estimator to be consistent is that its \textit{mean squared error} converges to $0$. The estimator is said to be \textit{asymptotically unbiased} if its mean converges to the true quantity to be estimated, i.e, $\lim_{k\to\infty}\|\mathbb{E}(e_k)\| = 0$. For a full exposition on the subject see \cite{casella}.

Theorem \ref{asymp stability theorem} asserts that for an LTI observable dynamical system with dynamic operator $A$ having all its eigenvalues outside the unit disk, the zero solution of the time varying system \eqref{eq_estabilidad_1} 
% $$
% z(k) = \Psi_k z(k-1)\ \  z(0) =z_0
% $$
is asymptotically stable. As it was already mentioned at the beginning of Section \ref{asymptotics}, this dynamical system is in fact the dynamical system of the expected value of the error estimation $\mathbb{E}(e_k)$. Then, under the aforementioned hypotheses we have that
\begin{equation*}
    \lim_{k\to \infty}\|\mathbb{E}(e_k)\| = \lim_{k\to\infty}\|P_k\| = 0.
\end{equation*}
Recall that the matrices $P_k$ given in \eqref{covariance} satisfy
$$P_k = \mathbb{E}(e_ke_k^*) - \mathbb{E}(e_k)\mathbb{E}(e_k)^*,
$$
then by taking traces in this expression we get
\begin{equation}\label{variance}
    \Tr P_k = \mathbb{E}(\|e_k\|^2) - \|\mathbb{E}(e_k)\|^2
\end{equation}
where we have used the identities 
\begin{align*}
    \Tr(\mathbb{E}(e_ke_k^*)) &=\mathbb{E}(e_k^*e_k)= \mathbb{E}(\|e_k\|^2)\\
    \Tr(\mathbb{E}(e_k)\mathbb{E}(e_k)^*) &= \mathbb{E}(e_k)^*\mathbb{E}(e_k) = \|\mathbb{E}(e_k)\|^2.
\end{align*}
The first term in the right hand side of \eqref{variance} is called the \textit{mean squared error} for the estimator $\widehat{x}_k$. In view of the above discussion we have the following result.

\begin{corollary}
For the observable dynamical system \eqref{dynSys} with $\lambda_{\min}(A)>1$, Algorithm \ref{alg-Kalman} gives an asymptotically unbiased and consistent estimator. 
\end{corollary}

\subsection{Numerical examples}
In this section we give two numerical examples to depict the asymptotic stability  and stability for the estimation error dynamics for the method proposed in Section \ref{sec_our_method}. In Example 1 we consider a dynamics satisfying the hypotheses of Theorem \ref{asymp stability theorem}. In Example 2 we have a dynamical system for which the error dynamics will be only Lyapunov stable.

\subsubsection{Example 1}
Consider the LTI system
$$\begin{cases}
    x(k+1) = Ax(k) \\
    y(k) = Hx(k) + v_k
\end{cases}$$
with 
$$
A = \begin{bmatrix} 1.99 &-0.32 & 0.  &  0.07\\
          0.43 & 1.17 & 0.02 & 0. \\
          0.13 & -0.09 & 1.52 & -0.13\\
          0.28 &-0.14 & 0.03 & 1.22\end{bmatrix}, \qquad   H = \begin{bmatrix}
         1 & 0& 0& 0\\
         0 & 0& 1& 0
    \end{bmatrix}
$$
and $v_k \sim N(0, \sigma\mathbb I)$ for  $\sigma =0.01$. The true initial state is $x_0 = [0.2, 0.4, 0.5, 0.3]$. 

Consider initial conditions 
$\widehat{x}_0 = [0.376, 0.502, 0.421, 0.366]$, initial error covariance $P_0 = 10^{-2}\mathbb{I}$. 

\begin{figure}[H]
\centering
\begin{subfigure}[b]{0.48\textwidth}
\centering
\includegraphics[width=1\linewidth, height=4cm]{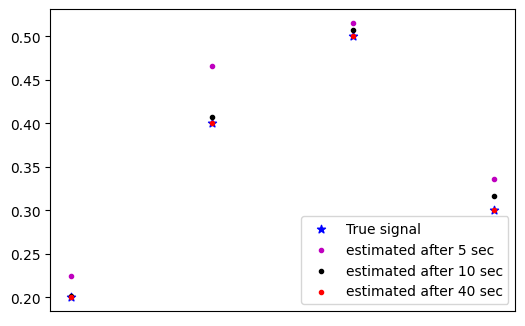} 
\caption{True and estimated coordinates of $x_0$.}
\label{fig:subim1}
\end{subfigure}
\begin{subfigure}[b]{0.5\textwidth}
\centering
\includegraphics[width=1\linewidth, height=4cm]{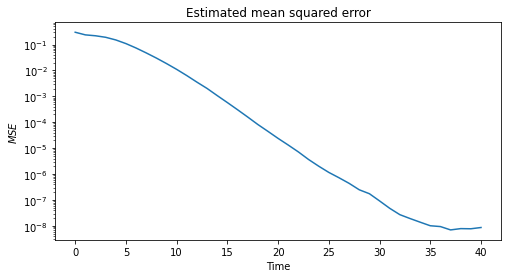}
\caption{Mean squared error.}
\label{fig:subim2}
\end{subfigure}
\begin{subfigure}[b]{1\textwidth}
\centering
\includegraphics[width =1\linewidth,height=5cm]{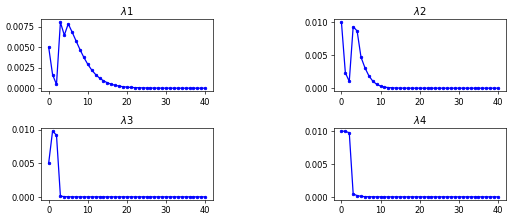}
\caption{Trajectory of the four eigenvalues of the matrices $P_k$.}
\label{fig:subim3}
\end{subfigure}

\caption{Error analysis for Example 1.}
\label{fig:image2}
\end{figure}
Figure (\ref{fig:subim1}) illustrates the plots of original signal (true signal) and its estimation after $5, 10, 40$ time steps. Figure (\ref{fig:subim2}) shows the evolution of the mean squared error $\mathbb{E}(e_k)$ for 40 time steps. Figure (\ref{fig:subim3}) shows the evolution for the four eigenvalues of the covariance matrix $P_k$ at every instant $k$ during 40 time steps.

\subsubsection{Example 2}
Consider the LTI system
$$\begin{cases}
    x(k+1) = Ax(k) \\
    y(k) = Hx(k) + v_k
\end{cases}$$
with 
$$
A = \begin{bmatrix} 1 &-0.5 \\
          -0.5 & 1 \end{bmatrix}, \qquad   H = \begin{bmatrix}
         0 & 1
    \end{bmatrix}
$$
and $v_k \sim N(0, \sigma\mathbb I)$ for  $\sigma =0.001$. The true initial state is $x_0 = [0.83053274, 0.35472554]$. 
The initialization for the algorithm is:
$\widehat{x}_0 = [0.99065169, 0.19889222]$, initial error covariance $P_0 = 10^{-2}\mathbb{I}$. 
\begin{figure}[H]
\centering
\begin{subfigure}[b]{0.48\textwidth}
\centering
\includegraphics[width=1\linewidth, height=4.2cm]{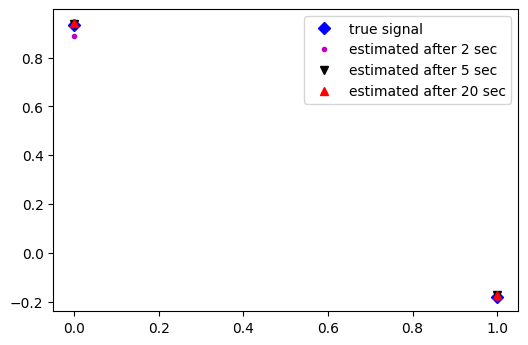} 
\caption{True and estimated coordinates of $x_0$.}
\label{fig:subim12}
\end{subfigure}
\begin{subfigure}[b]{0.5\textwidth}
\centering
\includegraphics[width=1\linewidth, height=4.2cm]{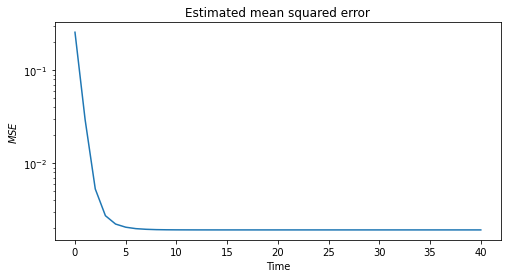}
\caption{Mean squared error.}
\label{fig:subim22}
\end{subfigure}
%\begin{subfigure}[b]{1\textwidth}
%\centering
%\includegraphics[width =0.9\linewidth,height=5cm]{Example2_eigenvals.png}
%\caption{Trajectory of the two eigenvalues of the matrices $P_k$}
%\label{fig:subim32}
%\end{subfigure}
\caption{Error analysis for Example 2.}
\label{fig:image22}
\end{figure}
Figure (\ref{fig:subim12}) illustrates the plots of original signal (true signal) and its estimation after $2, 5, 20$ time steps. Figure (\ref{fig:subim22}) shows the evolution of the mean squared error $\mathbb{E}(e_k)$ for 40 time steps. 

\end{document}